\documentclass[12pt,reqno]{amsart}
\usepackage[headings]{fullpage}
\usepackage{amssymb,amsmath,mathtools,bbm,tikz,tikz-cd,color,mathrsfs}
\usepackage[bookmarks=true,%
    colorlinks=true,%
    linkcolor=blue,%
    citecolor=blue,%
    filecolor=blue,%
    menucolor=blue,%
    urlcolor=blue,%
    breaklinks=true]{hyperref}

\newtheorem{theorem}{Theorem}[section]
\theoremstyle{definition}
\newtheorem{proposition}[theorem]{Proposition}
\newtheorem{lemma}[theorem]{Lemma}

\newtheorem{corollary}[theorem]{Corollary}

\def\a{\alpha}
\def\b{\beta}

\def\th{\theta}

\def\BZ{\mathbbm Z}
\def\BQ{\mathbbm Q}

\def\BC{\mathbbm C}

\def\id{\mathrm{id}}
\def\End{\mathrm{End}}
\def\Hom{\mathrm{Hom}}
\def\Tr{\mathrm{Tr}}

\def\mt{\mathsf{tr}}
\def\ptr{\mathrm{ptr}}
\def\md{\mathsf{d}}
\def\br#1{ \{ #1 \}}

\def\LG{\mathrm{LG}}
\def\CJ{\mathcal{J}}
\def\m{\widetilde{m}}
\def\A{\widetilde{A}}

\begin{document}

\title[The Colored Jones Function does not determine Links--Gould]{The Colored Jones Function does not determine any colored Links--Gould polynomial}

\author[M. Harper]{Matthew Harper}
\address{Department of Mathematics, Michigan State University, East Lansing, Michigan, USA}
\email{mrhmath@proton.me}
   
\author[B.-M. Kohli]{Ben-Michael Kohli}
\address{Section de Math\'ematiques, Universit\'e de Gen\`eve \\
rue du Conseil-G\'en\'eral 7-9, 1205 Gen\`eve, Switzerland}
\email{bm.kohli@protonmail.ch}


\date{October 8, 2026}

\begin{abstract}
We prove that there is no uniform description of $\mathrm{LG}^{(n)}$ in terms of colored Jones polynomials for all knots, even when allowing infinitely many colors and arbitrary operations. We prove the $n=1$ case, using an existing cabling formula and the separation of the Conway and Kinoshita--Terasaka knots by $\mathrm{LG}^{(2)}$. We develop an odd cabling formula and use a Vandermonde argument to show that for each $n\geq 2$ some satellite of each of these knots is distinguished by $\mathrm{LG}^{(n)}$.
\end{abstract}
\maketitle

\section{Introduction}
We provide an elementary proof that the colored Jones polynomials do not determine any colored Links--Gould polynomial. Our proof only requires the existence of a mutant pair that is separated by $\LG^{(2)}$. By the identification $V_2=\LG^{(2)}$ of \cite{GHKKW} the Conway and Kinoshita--Terasaka knots are known to be such a pair \cite{GL:data}.

Let $J^{(n)}_K(q)$ denote the $n$-th colored Jones polynomial of a knot $K$ in the variable $q$, associated to the $(n+1)$-dimensional representation of $U_q(\mathfrak{sl}_2)$ under the Reshetikhin--Turaev construction \cite{RT:ribbon}. Let 
\[
\CJ(K)=\{J^{(n)}_K(q)\}_{n=1}^\infty
\]
be the sequence of colored Jones polynomials of $K$, which we will refer to as the \emph{colored Jones function}.

The $n$-th colored Links--Gould polynomial $\LG^{(n)}(q^\a,q)$ is defined using the $4n$-dimensional Kac representation $V(n-1,\a)$ of quantum $\mathfrak{sl}(2|1)$ \cite{Links-Gould, DW01,GHKKST, GHKKW}. We refer the reader to these references and references therein for further background.

\begin{theorem}\label{thm}
    For every $n\geq 1$, there exist knots $K$ and $K'$ such that
    \[
    \CJ(K)=\CJ(K') \qquad \mbox{and} \qquad
    \LG_{K}^{(n)}(q^\a,q)\neq \LG_{K'}^{(n)}(q^\a,q).
    \] 
    
\end{theorem}
Let us describe witnesses to the theorem. For a knot $K$, write $K^{(2,p)}$ for its $(2,p)$-cabling, that is, the satellite knot obtained by 2-cabling a zero-framing of $K$ and applying $p$ twists. When $p$ is odd, the result is a knot. Throughout we use the notation
\begin{align*}
    C=\mathsf{11n_{34}}, && KT=\mathsf{11n_{42}}, && L=C^{(2,1)}, && M=KT^{(2,1)}.
\end{align*}
The knots $C$ and $KT$ are the Conway knot and Kinoshita--Terasaka knot mutant pair. These knots and their satellites are known to have equal colored Jones functions by joining results of \cite{MortonTraczyk,MortonStrickland}.  One could instead begin with any pair of knots with equal colored Jones functions that are distinguished by $\LG^{(2)}$. We prove the  $n=1$ case in Section \ref{sec.n=1}. This case is relatively straightforward and only uses the $(2,1)$-cabling formula proven in \cite{GHKKW} to show that the Links--Gould polynomial distinguishes $L$ from $M$.

The proof of the general case relies on a $(2,2j+1)$-cabling formula for $\LG^{(n)}$, which we prove in Section \ref{sec.oddcable}, building on the representation theoretic discussion in Section \ref{sec.reps}. Our proof shows that the pair $K=L^{(2,2j+1)}$ and $K'=M^{(2,2j+1)}$ for some $0\leq j \leq 3n-1$ is a witness to Theorem \ref{thm}. With additional work, one may be able to determine exactly which cablings of $L$ and $M$ to use for each $n$.  

As an immediate corollary, no colored Links--Gould polynomial $\LG^{(n)}$ is a function of the colored Jones polynomials.
\begin{corollary}\label{cor}
    For every integer $n\geq 1$, there does not exist a function $\Phi_n$ on the image of the colored Jones function such that for all knots $K$,
    \[
    \LG^{(n)}_K(q^\a,q)=\Phi_n(\CJ(K)).
    \]
\end{corollary}

We prove an additional corollary in Section \ref{sec.cor2}.
\begin{corollary}\label{cor2}
    For every $n\geq 1$ and $N\geq2$, there exist knots $K_1,\dots, K_N$ with equal colored Jones sequences $\CJ(K_1)=\dots = \CJ(K_N)$ but their $n$-th colored Links--Gould polynomials $\LG_{K_1}^{(n)}, \dots, \LG_{K_N}^{(n)}$ are pairwise distinct.
\end{corollary}

% \blue{An interesting question arising from the tools and techniques developed here is the following.}

% \begin{question}
%     Does $\LG$ (or any $\LG^{(n)}$) specialize to the Jones polynomial? 
% \end{question}

% \blue{While it is well known that $\LG^{(n)}$ specializes to the Alexander polynomial CITE, the authors believe that the answer to Question REF is likely negative. The results of this paper, together with the existence of genus bounds for $\LG^{(n)}$ and the absence of analogous bounds for the Jones polynomial, provide two strong indications in this direction. However, a formal proof of this statement remains to be found.}\\

While this paper was being finalized, a proof of Theorem \ref{thm} in the $n=1$ case was posted by Garoufalidis and Li \cite{GL:JonesLG}. 

\subsection*{Acknowledgements}
MH was partially supported through the NSF-RTG grant
\#DMS-2135960.

\section{Colored Jones and Links--Gould of the mutant pair}\label{sec.n=1}
We prove the $n=1$ case of Theorem \ref{thm} and Corollary \ref{cor}, which is much simpler than for general $n$. For $n=1$, we write $\LG=\LG^{(1)}$.

Since the Conway knot and Kinoshita--Terasaka knot are related by mutation, by \cite[Thm.2]{MortonTraczyk}, any pair of links obtained by applying the same satellite operation to them have equal Jones polynomials. Furthermore, Morton-Strickland prove that $\CJ$ of a framed link is determined by the Jones polynomial $J=J^{(1)}$ of its cablings \cite[Thm.2.5,3.2]{MortonStrickland}.  Combining these results, the colored Jones functions do not distinguish the images of $C$ and $KT$ under any satellite operation, in particular under iterated satellite operations. 
\begin{lemma}\label{lem.CJ}
    For any integer $j$, the following equalities hold
    \[
    \CJ(L)=\CJ(M) \qquad \mbox{and} \qquad
    \CJ(L^{(2,2j+1)})=\CJ(M^{(2,2j+1)}).
    \]
\end{lemma}

The ordinary Links--Gould polynomial does not detect mutation, due to the tensor product $V(0,\a)^{\otimes 2}$ being multiplicity free \cite[Thm.5]{MortonCromwell} with \cite{GHKKW}, also see \cite[Sec.4.8]{DWKL}. However, $\LG^{(2)}$ does detect mutation and indeed distinguishes $C$ and $KT$. This is observed using the data of Garoufalidis--Li and the identification of the $V_2$-polynomial with $\LG^{(2)}$ \cite{GL:data,GHKKW}. 

Recall the $(2,1)$-cabling formula for a $0$-framed knot $K$ proven in \cite[Rem.4.1]{GHKKW}:
\begin{gather}\label{eq.LGcable}
    \begin{aligned}
    \LG_{K^{(2,1)}}(q^\a,q) = \,&
    q^{2\a} A^{(2)}_{0,0}(q^\a,q) \,\LG_{K}(q^{2\a},q) 
    +
    q^{-2(\a+1)} A^{(2)}_{0,1}(q^\a,q)\, \LG_{K}(q^{2\a+1},q)
\\
  &-A^{(2)}_{1,0}(q^\a,q)\, \LG^{(2)}_K(q^{2\a},q).
    \end{aligned}
\end{gather}
The coefficient
\begin{align*}
    A^{(2)}_{1,0} = -\frac{\br{2}\br{\a}\br{\a+1}}{\br{1}\br{2\a}\br{2\a+2}} \qquad \mbox{where $\br{x}=q^x-q^{-x}$}.
\end{align*}

\begin{lemma}
    There does not exist a function $\Phi$ on the set of $\BZ[q,q^{-1}]$-sequences such that
    \[
    \LG_K(q^\a,q)=\Phi(\CJ(K))
    \]
    for all $K$. In particular, $\LG$ is not determined by $\CJ$. 
\end{lemma}
\begin{proof}
    It suffices to exhibit $L=C^{(2,1)}$ and $M=KT^{(2,1)}$ as a pair for which the colored Jones functions are equal, but which are distinguished by the Links--Gould polynomial. The first claim is exactly the statement of Lemma \ref{lem.CJ}. Now consider their difference of Links--Gould polynomials and apply \eqref{eq.LGcable},
    \[
    \LG_L(q^\a,q)-\LG_M(q^\a,q)=
    -A^{(2)}_{1,0}(q^\a,q)\left(\LG^{(2)}_C(q^{2\a},q)-\LG^{(2)}_{KT}(q^{2\a},q)\right).
    \]
    Since $\LG^{(2)}$ distinguishes $C$ from $KT$, and $A^{(2)}_{1,0}$ is nonzero, $\LG$ distinguishes $L$ and $M$.
\end{proof}

\section{Representation theory for \texorpdfstring{$U_q^H(\mathfrak{sl}(2|1))$}{Uqsl(2|1)}}\label{sec.reps}
In this section we discuss properties of the tensor square of representations defining $\LG^{(n)}$ for $n\geq 2$. We establish the multiplicities of summands and notation for the projectors and inclusions realizing the direct sum.

\subsection{The representation \texorpdfstring{$V$}{V} and its tensor square} 
Fix $n\geq 2$ and set 
\[V=V_{n-1}(n-1,\a)=V(n-1,\a)\]
the representation associated to $\LG^{(n)}$ in the variables $(q^\a,q)$. Here $V_p(n-1,\a)$ is the $4n$-dimensional representation of $U_q^H(\mathfrak{sl}(2|1))$ with highest weight $(n-1,\a)\in \BZ\times \BC$ and parity of the highest weight vector $p\mod 2$. Throughout we will assume that $(n-1,\a)$ is a typical weight, i.e. $\a(n+\a)\neq0$, so that $V$ is irreducible. We further suppose that $\a\notin\frac{1}{2}\BZ$ so that the later tensor product decomposition holds. We will suppress the parity in our notation for the modules and for $\LG$ as in \cite{GHKKW}. We recall the formulae for the modified dimension and twist on $V$
\begin{align}\label{eq.mdth}
    \md(V)=(-1)^{n-1}\frac{\{n\}}{\{\alpha\}
    \{n+\alpha\}} && \mbox{and} && \th_V=q^{-2\a(n+\a)}
\end{align}
from \cite[Eqn.(84), Proof of Prop.3.14]{GHKKW}. 
We note the slight change in notation compared to \textit{loc.~cit.}, here $\th_V$ indicates the scalar value of the twist morphism in $\End(V)$ rather than the morphism itself.

Let $\m^{(n)}_{k,\ell}$ denote the multiplicity of $V(k,2\a+\ell)$ in $V\otimes V$. In other words,
\[
V(n-1,\a)\otimes V(n-1,\a)\cong \bigoplus_{k,\ell\geq0} \m^{(n)}_{k,\ell} \, V_k(k,2\a+\ell).
\]
With this fixed grading convention, we assume that $V_k(k,2\a+\ell)=V(k,2\a+\ell)$ has parity $k\mod 2$.
We caution the reader not to confuse $\m$ with  $m$ used in \cite{GHKKW} to denote the multiplicity of $V(k,n\a+l)$ in $V(0,\a)^{\otimes n}$. For later use, we also define an analog of $A^{(n)}_{k,\ell}$ for the ratio of modified dimensions between summands and $V$. Set
\begin{equation}\label{eq.A}
    \begin{aligned}
        \A^{(n)}_{k,\ell}=\frac{\md(V(k,2\a+\ell))}{\md(V(n-1,\a))}
        =(-1)^{n-k-1}\frac{\br{k+1}\br{\a}\br{\a+n}}{\br{n}\br{2\a+\ell}\br{2\a+k+\ell+1}}.
    \end{aligned}
\end{equation}
Notice that $\A^{(n)}_{0,n-1}$ is nonzero.

\begin{lemma}\label{lem.fusion}
    The nonzero values of $\m^{(n)}_{k,\ell}$ are given by Table \ref{table.fusion}. In particular, there are $3n$ isotypic components and $V(0,2\a+n-1)$ has multiplicity one.
\end{lemma}
\begin{table}[h!]
\centering
\[
\begin{array}{cccc}
k & \ell & \text{range} & \widetilde m^{(n)}_{k,\ell} \\
\hline
2r   & n-1-r & 0\leq r\leq n-1 & 1 \\
2r   & n-r   & 0\leq r\leq n-1 & 1 \\
2r+1 & n-1-r & 0\leq r\leq n-2 & 2 \\
2n-1 & 0     & - & 1
\end{array}
\]
\caption{Nonzero multiplicities of \(V(k,2a+\ell)\)
in \(V_{n-1}(n-1,a)^{\otimes 2}\).}
\label{table.fusion}
\end{table}

\begin{proof}
    We refer to the decomposition rule \cite[Lem.3.9]{GHKKW}. Although it is stated for $\a$, $\b$, and $1$ being $\BZ$-linearly independent, it holds even when $\a=\b$ under our standing assumption $\a\notin \frac{1}{2}\BZ$ so that all weights are typical. In our case, the isomorphism is 
    \begin{equation*}
        \begin{aligned}
        V(n-1,\a)^{\otimes 2}\cong&~ 
V(2n-1,2\a) \\
  &\oplus \bigoplus_{k=0}^{2n-2} V\left(k, 2\a+n-1+
\lfloor \tfrac{2-k}{2} \rfloor\right) \oplus
V\left(k, 2\a+n-1 + \lfloor \tfrac{1-k}{2} \rfloor\right).
    \end{aligned}
    \end{equation*}
    There is a unique maximum $k=2n-1$ summand $V(2n-1,2\a)$. For even indices $k=2r$, where the integer $r$ varies from $0$ to $n-1$, the summands are  
    \[
    V(2r, 2\a+n-r) \oplus
V(2r, 2\a+n-r-1).
    \]
     For odd indices $k=2r+1$ where $r$ varies from $0$ to $n-2$, the summands are 
    \[
    V(2r+1, 2\a+n-1-r) \oplus
V(2r+1, 2\a+n-1 -r).
    \]
    The table organizes these components.
\end{proof}

\subsection{Projectors and inclusions} 
Fix $(k,\ell)$ and set $U=V(k,2\a+\ell)$. Let
\[
M^{(n)}_{k,\ell}=\Hom(V(k,2\a+\ell), V(n-1,\a)^{\otimes 2}) \qquad \mbox{with}\qquad \dim M^{(n)}_{k,\ell} = \m^{(n)}_{k,\ell}
\]
be the multiplicity space of $U$ in $V^{\otimes 2}$. With the exception of the odd $k$ summands of the decomposition, the multiplicity spaces are 1-dimensional. 
Define inclusion and projection maps
\begin{equation*}
    i_s:U\to V^{\otimes 2}, \qquad p_s:V^{\otimes 2}\to U,\qquad 
    p_t\circ i_s =\delta_{t,s}\id_{U}, \qquad 1\leq t,s\leq \m^{(n)}_{k,\ell}.
\end{equation*}
Although it is not explicit in the notation, a choice of $(k,\ell)$ will always be assumed when considering the maps $i_s$ and $p_s$. For any $U_q^H(\mathfrak{sl}(2|1))$-intertwiner $f\in \End(V^{\otimes 2})$, the action of $f$ on the multiplicity space $M^{(n)}_{k,\ell}$ is denoted $f_{k,\ell}$. The $(s,t)$-matrix entry of $f_{k,\ell}$ with respect to the basis induced by the $\{i_s\}_s$ is 
\[
(f_{k,\ell})_{s,t}=p_s\circ f\circ i_t.
\]
We will also write $P_{(k,\ell)}=\sum_{s} i_s\circ p_s$ for the idempotent projection on $V^{\otimes 2}$ which factors through $U^{\oplus \m^{(n)}_{k,\ell}}$. We note the channeling identity 
\begin{align}\label{eq.channel}
    \id_{V^{\otimes 2}}=\sum_{k,\ell\geq0} P_{(k,\ell)}.
\end{align}

\section{Proof of Theorem \texorpdfstring{\ref{thm}}{\ref*{thm}} for \texorpdfstring{$n\geq 2$}{n>=2}}
The proof of the theorem is organized into two parts. First, we compute the colored Links--Gould polynomials of a $(2,2j+1)$-cabling of a knot $K$ in terms of colored Links--Gould polynomials of $K$. Second, we construct a linear system associated to the isotypic components of $V^{\otimes2}$ and the odd-cabling factors and determine that it does {not admit} a solution where $\LG^{(n)}_{L^{(2,2j+1)}}=\LG^{(n)}_{M^{(2,2j+1)}}$ for all $j$ between $0$ and $3n-1$. The proof applies uniformly for all $n\geq 2$ and is not inductive.

\subsection{Odd-cabling formula} \label{sec.oddcable} We apply the representation theory of the previous section to establish properties of the braiding and RT morphisms. In this section, we prove a general form for the Links--Gould polynomial of a $(2,2j+1)$-cabling. 
We carry the notation $U=V(k,2\a+\ell)$ from the previous section.

Let $c=c_{V,V}$ be the braiding on $V\otimes V$. Let $K$ be a blackboard-framed knot diagram with zero writhe, representing its zero-framing. Let $T$ be a $(1,1)$-tangle representative of $K$ and $F_T\in\End(V^{\otimes 2})$ the Reshetikhin--Turaev operator for $T$ colored by $V^{\otimes 2}$. If we instead consider $T$ colored by the representation $U$, then we write $F_{U,T}$. By definition
\begin{align*}
    F_{U,T} = \LG^{(k+1)}(q^{2\a+\ell},q) \, \id_U.
\end{align*}

\begin{lemma}\label{lem.FTisot}
    The action of $F_T$ on each $V(k,2\a+\ell)$-isotypic component of $V^{\otimes 2}$ is by the scalar $\LG^{(k+1)}(q^{2\a+\ell},q)$. Equivalently,
    \[
    F_T\circ P_{(k,\ell)}=P_{(k,\ell)}\circ F_T = \LG^{(k+1)}(q^{2\a+\ell},q)\cdot P_{(k,\ell)}
    \]
\end{lemma}
\begin{proof}
    We apply naturality of the RT construction on each inclusion and projection map, see for example \cite[Lem.2.2]{GHKKW}. Then for each $s$ and $t$, 
    \[
    F_T\circ i_s = i_s\circ F_{U,T}\qquad \mbox{and} \qquad p_t \circ F_T= F_{U,T}\circ p_t.
    \]
    Hence, we have the scalar action
    \[
    (F_T)_{s,t}=p_s\circ F_T\circ i_t=\delta_{s,t}\, \LG^{(k+1)}(q^{2\a+\ell},q).
    \]
    The second claim follows from expanding $P_{(k,\ell)}=\sum_{s} i_s\circ p_s$ and again applying naturality.
\end{proof}

We now give a description of the odd-braiding on multiplicity space.
\begin{lemma}\label{lem.cisot}
    For each $V(k,2\a+\ell)$-isotypic component of $V^{\otimes 2}$, the braiding satisfies
    \[
    c_{k,\ell}^{2j+1}=\lambda_{k,\ell}^j c_{k,\ell} \qquad \mbox{where $\lambda_{k,\ell}=\dfrac{\th_{V(k,2\a+\ell)}}{\th_{V}^2}$.}
    \]
\end{lemma}
We remark that the matrix $c_{k,l}$ need not be scalar when $k$ is odd.
\begin{proof}
    As before, set $U=V(k,2\a+\ell)$. For each $s$, we have
    \[
    \th_{V\otimes V}i_s=i_s\th_{U} \qquad \mbox{and} \qquad c^2=\th_{V\otimes V}(\th_V^{-1}\otimes \th_V^{-1})\,\id_{V^{\otimes 2}} = \th_V^{-2}\th_{V\otimes V}\,\id_{V^{\otimes 2}}.
    \]
    Since $c$ is an intertwiner, we may write 
    \[
    (c_{k,\ell}^2)_{s,t}= p_s \circ c^2\circ i_t = \delta_{s,t}\,\frac{\th_U}{\th_V^2}=\delta_{s,t}\,\lambda_{k,\ell}.
    \]
    Taking $j$-th powers and multiplying by $c_{k\ell}$ yields the desired identity.
\end{proof}

\begin{proposition}\label{prop.cable}
    Let $K$ be a zero-framed knot. Then for every $j\geq0$,
    \[
    \LG_{K^{(2,2j+1)}}(q^\a,q)=\th_V^{-(2j+1)}\sum_{k,\ell\geq 0} \lambda_{k,\ell}^j\,\A_{k,\ell}^{(n)}(q^\a,q)\,  \Tr_{M^{(n)}_{k,\ell}}(c_{k,\ell})\, \LG_K^{(k+1)}(q^{2\a+\ell},q),
    \]
    where $\Tr_{M^{(n)}_{k,\ell}}$ is the trace of linear maps on $M^{(n)}_{k,\ell}$ as a vector space. Moreover, for $(k, \ell) = (0, n-1)$, the trace $\Tr_{M^{(n)}_{k,\ell}}(c_{k,\ell})$ is nonzero.
\end{proposition}
\begin{proof}
    Let $T$ be a $(1,1)$-tangle representative of $K$ and $F_T$ the RT operator for $T$ colored by $V^{\otimes 2}$. Equivalently, we may consider the $(2,0)$-cabling of $T$ and color it by $V$. The $(2,2j+1)$-cabling of $K$ is then presented by $F_T\circ c^{2j+1}$ and its modified trace satisfies
    \begin{align}\label{eq.trtr}
        \mt_{V\otimes V}(F_T\circ c^{2j+1}) = \mt_V(\ptr_V(F_T\circ c^{2j+1})).
    \end{align}
    We proceed by evaluating each side of the equality.

    For the left-hand side, we apply the channeling identity \eqref{eq.channel}, followed by Lemmas \ref{lem.FTisot} and \ref{lem.cisot}, and cyclicity of the modified trace
    \begin{equation*}
        \begin{aligned}
            \mt_{V\otimes V}(F_T\circ c^{2j+1}) &=\sum_{k,\ell\geq0}\mt_{V\otimes V}(F_T\circ c^{2j+1}\circ P_{(k,\ell)})
            \\
            &=
            \sum_{k,\ell\geq0} \LG_K^{(k+1)}(q^{2\a+\ell},q)\, \lambda_{k,\ell}^j \mt_{V\otimes V}(c\circ P_{(k,\ell)})
            \\
            &=
            \sum_{k,\ell\geq0} \LG_K^{(k+1)}(q^{2\a+\ell},q)\, \lambda_{k,\ell}^j \sum_s \mt_{U}(p_s \circ c \circ i_s)
            \\
            &=
            \sum_{k,\ell\geq0} \LG_K^{(k+1)}(q^{2\a+\ell},q)\, \lambda_{k,\ell}^j \sum_s c_{ss} \mt_{U}(\id_U)
            \\
            &=
            \sum_{k,\ell\geq0} \md(V(k,2\a+\ell))\, \LG_K^{(k+1)}(q^{2\a+\ell},q)\, \lambda_{k,\ell}^j \Tr_{M^{(n)}_{k,\ell}}(c_{k,\ell}).
        \end{aligned}
    \end{equation*}
    To resolve the right-hand side of \eqref{eq.trtr}, we adjust the $(1,1)$-tangle diagram associated to $\ptr_V(F_T\circ c^{2j+1})$ so that it has zero writhe by introducing a framing correction. We have
    \[
    \ptr_V(F_T\circ c^{2j+1}) = \th_V^{2j+1}\ptr_V(\th_V^{-(2j+1)}\, F_T\circ c^{2j+1})= \th_V^{2j+1}\,\LG_{K^{(2,2j+1)}}^{(n)}(q^\a,q) \,\id_{V}
    \]
    which when composed with $\mt_V$ gives $\th_V^{2j+1}\,\LG_{K^{(2,2j+1)}}^{(n)}(q^\a,q) \,\md(V)$.

    Combining both sides, and recalling $\A^{(n)}_{k,\ell}$ from \eqref{eq.A}, we now have the desired identity
    \begin{align*}
        \th_V^{2j+1}\,\LG_{K^{(2,2j+1)}}^{(n)}(q^\a,q)\, \md(V) &= \sum_{k,\ell} \md(V(k,2\a+\ell))\, \LG_K^{(k+1)}(q^{2\a+\ell},q)\, \lambda_{k,\ell}^j \Tr_{M^{(n)}_{k,\ell}}(c_{k,\ell})
        \\
        \LG_{K^{(2,2j+1)}}^{(n)}(q^\a,q) &=\th_V^{-(2j+1)} \sum_{k,\ell\geq0} \lambda_{k,\ell}^j \frac{\md(V(k,2\a+\ell))}{\md(V)} \Tr_{M^{(n)}_{k,\ell}}(c_{k,\ell}) \LG_K^{(k+1)}(q^{2\a+\ell},q).
    \end{align*}

    To prove the latter claim, we recall from Lemma \ref{lem.fusion} that $V(0,2\a+n-1)$ has multiplicity one. Since the braiding is invertible, each matrix $c_{k,\ell}$ must be invertible. In particular, $\Tr_{M^{(n)}_{0,n-1}}(c_{0,n-1})=c_{0,n-1}$ is invertible and hence nonzero.
\end{proof}

\subsection{System of twists}
We complete the proof of Theorem \ref{thm}. We begin by proving that the $\lambda_{k,\ell}$ are pairwise distinct.

\begin{lemma}\label{lem.lambda}
    For $(k,\ell)$ appearing in Table \ref{table.fusion}, each of the $3n$ numbers $\lambda_{k,\ell}$ is given by the expression
    \[
    \lambda_{k,\ell}=q^{-4\a^2+4\a(n-1-k-2\ell)-2\ell(k+\ell+1)}
    \] 
    and are pairwise distinct.
\end{lemma}
\begin{proof}
    Evaluating $\lambda_{k,\ell}$, defined in Lemma \ref{lem.cisot}, with $\th_V$ as in \eqref{eq.mdth} gives the prescribed value.
    We check the coefficient of $\a$ which is determined by $k+2\ell$. In Table \ref{table.fusion}, the value $k+2\ell$ in each row is $2n-2$, $2n$, $2n-1$, and again $2n-1$. Within each row, or the last pair of rows, we consider the constant term in the exponent of $q$, determined by $\ell(k+\ell+1)$. We verify that these are row-wise distinct for each pair $(k,\ell)$.
    \[
    \begin{array}{ccccc}
        k & \ell & \text{range} & k+2\ell& \ell(k+\ell+1)\\
        \hline
        2r   & n-1-r & 0\leq r\leq n-1 & 2n-2& n(n-1)-r(r+1)\\
        2r   & n-r   & 0\leq r\leq n-1 & 2n& n(n+1)-r(r+1) \\
        2r+1 & n-1-r & 0\leq r\leq n-2 & 2n-1& n^2-(r+1)^2 \\
        2n-1 & 0     & - &2n-1& 0
        \end{array}
    \]
    Note that $\ell(k+\ell+1)>0$ in the third row and is distinct from the value $0$ in the fourth row.
\end{proof}

We now have the machinery prepared to complete the proof of the main theorem.
\begin{proof}[Proof of Theorem \ref{thm} for $n\geq 2$]
    Suppose for the sake of a contradiction that 
    \[
    \LG^{(n)}_{L^{(2,2j+1)}}(q^\a,q)=\LG^{(n)}_{M^{(2,2j+1)}}(q^\a,q)
    \]
    for all integers $0\leq j \leq 3n-1$. For each $j$, applying Proposition \ref{prop.cable} and taking the difference of these polynomials implies
    \[
    0=\th_V^{-(2j+1)}\sum_{k,\ell\geq0} \lambda_{k,\ell}^j\A_{k,\ell}^{(n)}\Tr_{M^{(n)}_{k,\ell}}(c_{k,\ell})\left(\LG_{L}^{(k+1)}(q^{2\a+\ell},q)-\LG_{M}^{(k+1)}(q^{2\a+\ell},q)\right).
    \]
    Normalize out the ribbon factor and set 
    \[D_{k,\ell}=\A_{k,\ell}^{(n)}\Tr_{M^{(n)}_{k,\ell}}(c_{k,\ell})\left(\LG_{L}^{(k+1)}(q^{2\a+\ell},q)-\LG_{M}^{(k+1)}(q^{2\a+\ell},q)\right)\]
    so that over all $j$ we have a $3n\times 3n$ system
    \[
    \sum_{k,\ell\geq 0 } D_{k,\ell}\lambda^j_{k,\ell}=0.
    \]
    Choose an ordering of the isotypic components $V(k,2\a+\ell)=V_i$ and assign the same ordering to $\lambda_{k,\ell}=\lambda_i$ and $D_{k,\ell}=D_i$. This realizes a Vandermonde system
    \[
\begin{pmatrix}
1 & 1 & \cdots & 1 \\
\lambda_1 & \lambda_2 & \cdots & \lambda_{3n} \\
\vdots & \vdots & \ddots & \vdots \\
\lambda_1^{3n-1} & \lambda_2^{3n-1} & \cdots & \lambda_{3n}^{3n-1}
\end{pmatrix}
\begin{pmatrix}
D_1 \\
D_2 \\
\vdots \\
D_{3n}
\end{pmatrix}=0.
\]
The determinant of the matrix of $(\lambda_i^j)_{ij}$ is $\prod_{i<j}(\lambda_j-\lambda_i)$, which is nonzero by Lemma \ref{lem.lambda}. Hence, $D_i=0$ for all $i$.

Now consider $(k,\ell)=(0,n-1)$. Recall that $\A^{(n)}_{k,\ell}$ is nonzero and from Proposition \ref{prop.cable} that $\Tr_{M^{(n)}_{k,\ell}}(c_{k,\ell})$ is nonzero as well. 
Finally, the Links--Gould polynomial in $(q^\a,q)$ distinguishes $L$ from $M$.
Hence, for $(k,\ell)=(0,n-1)$,
\[
D_{k,\ell}=\A_{k,\ell}^{(n)}\Tr_{M^{(n)}_{k,\ell}}(c_{k,\ell})\left(\LG_{L}^{(k+1)}(q^{2\a+\ell},q)-\LG_{M}^{(k+1)}(q^{2\a+\ell},q)\right)\neq 0
\]
and we have reached a contradiction. There must be some $0\leq j\leq 3n-1$ such that $\LG^{(n)}_{L^{(2,2j+1)}}(q^\a,q)\neq\LG^{(n)}_{M^{(2,2j+1)}}(q^\a,q)$. Moreover, $\CJ(L^{(2,2j+1)})=\CJ(M^{(2,2j+1)})$ by Lemma \ref{lem.CJ}.
\end{proof}

\section{Proof of Corollary \texorpdfstring{\ref{cor2}}{\ref*{cor2}}}
\label{sec.cor2}
To prove the corollary we clarify our conventions for the colored Jones and Links--Gould polynomials. We assume that $J^{(n)}_{\mathsf{0_1}}(q)=1$ and $\LG^{(n)}_{\mathsf{0_1}}(q^\a,q)=1$ for all $n\geq 1$. These assumptions match those of \cite{GHKKW}. 

With these conventions these polynomials are multiplicative under connected sums
\[
J^{(n)}_{K\#K'}=J^{(n)}_{K}J^{(n)}_{K'} \qquad \mbox{and}
\qquad
\LG^{(n)}_{K\#K'}=\LG^{(n)}_K\LG^{(n)}_{K'}.
\]
Indeed, these invariants are computed on knots by computing the scalar associated to a $(1,1)$-tangle representative under the Reshetikhin--Turaev functor. Connected sums are represented by the stacking of $(1,1)$-tangles computing these invariants and hence by multiplication of scalars. 

\begin{proof}[Proof of Corollary \ref{cor2}]
    Fix $n\geq 1$. Let $A$ and $B$ be a pair of knots witnessing Theorem \ref{thm} so that $\CJ(A)=\CJ(B)$ and $\LG^{(n)}_A(q^\a,q)\neq\LG^{(n)}_{B}(q^\a,q)$. Let $K_m=A^{\# m-1}\#B^{\# N-m}$ for $1\leq m\leq N$. Multiplicativity of the Jones polynomial implies $\CJ(K_1)=\dots =\CJ(K_N)$.

    Similarly, multiplicativity of Links--Gould polynomials under connected sums gives
    \[
    LG^{(n)}_{K_m}(q^\a,q)=(\LG^{(n)}_{A}(q^\a,q))^{m-1}\cdot (\LG^{(n)}_{B}(q^\a,q))^{N-m}
    \] 
    Now \cite[Thm.1.3]{GHKKW} states $\LG^{(n)}_K(1,q)=1$ for every knot $K$. In particular, $\LG^{(n)}_A(q^\a,q)$ and $\LG^{(n)}_B(q^\a,q)$ are nonzero. 
    
    Suppose for contradiction that $LG^{(n)}_{K_s}=LG^{(n)}_{K_t}$ for some $s\neq t$. Let $x=\LG^{(n)}_{A}\cdot (\LG^{(n)}_{B})^{-1}$. By our assumption 
    \[
    \frac{(\LG^{(n)}_{A}(q^\a,q))^{s-1}\cdot (\LG^{(n)}_{B}(q^\a,q))^{N-s}}{(\LG^{(n)}_{A}(q^\a,q))^{t-1}\cdot (\LG^{(n)}_{B}(q^\a,q))^{N-t}} = x^{s-t}=1.
    \]
    This implies $x$ is a root of unity in $\BQ(q^\a,q)$ and hence $x=\pm 1$. If $x=1$, then this contradicts $\LG^{(n)}_A(q^\a,q)\neq\LG^{(n)}_B(q^\a,q)$. If $x=-1$, then this violates the specialization property $\LG^{(n)}_K(1,q)=1$. Therefore the colored Links--Gould polynomials of $K_1,\dots, K_N$ are pairwise distinct.
\end{proof}

\section*{Tool Use}No AI tools have been used to write or create this work.

\bibliographystyle{hamsalpha}
\bibliography{biblio}
\end{document}